\documentclass[a4paper,10pt]{article}
\usepackage[utf8]{inputenc}
\usepackage{amssymb}
\usepackage{amsmath}
\usepackage{enumitem}
\usepackage{amsthm}
\usepackage{siunitx}
\usepackage{comment}
\usepackage{indentfirst}
\usepackage{color,soul}
\usepackage[citation-order]{amsrefs}
\newenvironment{rezabib}
  {\bibdiv\biblist\setupbib}
  {\endbiblist\endbibdiv}

\def\setupbib{\catcode`@=\active}
\begingroup\lccode`~=`@
  \lowercase{\endgroup\def~}#1#{\gatherkey{#1}}
\def\gatherkey#1#2{\gatherkeyaux{#1}#2\gatherkeyaux}
\def\gatherkeyaux#1#2,#3\gatherkeyaux{\bib{#2}{#1}{#3}}

\newtheorem*{corollary*}{Corollary}
\newtheorem*{theorem*}{Theorem}

\theoremstyle{definition}
\newtheorem{theorem}{Theorem}
\newtheorem{lemma}[theorem]{Lemma}
\newtheorem{corollary}[theorem]{Corollary}
\newtheorem{definition}[theorem]{Definition}
\newtheorem{proposition}[theorem]{Proposition}
\newtheorem{problem}[theorem]{Problem}

\theoremstyle{remark}

\theoremstyle{remark}
\newtheorem{example}{Example}

\theoremstyle{remark}
\newtheorem*{example*}{Example}

\theoremstyle{remark}
\newtheorem*{remark}{Remark}

\theoremstyle{remark}

\newcommand{\A}{\mathcal{A}}
\newcommand{\J}{\mathcal{J}}
\newcommand{\C}{\mathcal{C}}
\newcommand{\D}{\mathcal{D}}

\newcommand{\Jq}{\J^{(q)}}
\newcommand{\Aq}{\A^{(q)}}
\newcommand{\Cq}{\C^{(q)}}
\newcommand{\Dq}{\D^{(q)}}

\title{Balancedly splittable orthogonal designs and equiangular tight frames}
\author{
 Hadi Kharaghani\thanks{Department of Mathematics and Computer Science, University of Lethbridge,
Lethbridge, Alberta, T1K 3M4, Canada. \texttt{kharaghani@uleth.ca}}
\and
Thomas Pender\thanks{Department of Mathematics and Computer Science, University of Lethbridge,
Lethbridge, Alberta, T1K 3M4, Canada. \texttt{thomas.pender@uleth.ca}}
\and
  Sho Suda\thanks{Department of Mathematics,  National Defense Academy of Japan, Yokosuka, Kanagawa 239-8686, Japan. \texttt{ssuda@nda.ac.jp}}
}
\date{\today}

\begin{document}
\maketitle
\begin{abstract}
The concept of balancedly splittable orthogonal designs is introduced along with a recursive construction.
As an application, equiangular tight frames over the real, complex, and quaternions meeting the Delsarte-Goethals-Seidel upper bound are obtained.
\end{abstract}

\section{Introduction}
A \emph{Hadamard matrix} is a square matrix $H_n$ of order $n$, with entries in the set $\{1,-1\}$, such that $H_nH_n^t = nI_n$. These structures, along with their more general counterpart \emph{weighing matrices}, have proved to be useful in a variety of settings, from theoretical applications in finite geometry and combinatorial designs, to applied applications in statistics and optics (see \cite{Seb} and the references cited therein). Despite the growing interest in these objects, insight into their existence and properties has proven difficult. One must draw on several branches of mathematics to answer even the simplest of questions.

Nevertheless, researchers have continued to find new and interesting applications for these elusive objects. A recent development in \cite{KharSuda2} has been to \emph{split} the matrix into two parts. A Hadamard matrix is said to be \emph{balancedly splittable} if it is permutation equivalent to a matrix that has an $\ell\times n$ submatrix $H_1$ such that $H_1^tH_1 = \ell I_n + aA + b\bar{A}$, where $A$ is the adjacency matrix of some simple (no loops) undirected graph and $\bar{A}$ its compliment. 
The aim in this paper is twofold; first, it generalizes and extends the existing results on balancedly splittable Hadamard matrices to larger classes. Second, it introduces variables in the structures to help open the concept for further research. A connection with a recent result of Shayne Waldron \cite{waldron2020tight} is discussed and examples are given.
Due to the variety of applications of equiangular tight frames in areas such as error-correcting codes, signal reconstruction, and other practical applications, see \cite{CCHT2018} for details, the attempt here is to add more useful tools to study them. 

Preliminary materials are covered in Section 2. The study of splittable orthogonal designs is divided to the \emph{faithful} and \emph{unfaithful} parts in Section 3. Section 4 is devoted to the construction of unfaithful case, and Section 5 to the faithful case. The connection to unbiased orthogonal designs is discussed in Section 6.

\section{Preliminaries}
\subsection{Quaternionic and Butson Hadamard matrices}
The constructions in Section 4 and 5 apply to real, quaternary, Butson, and Quaternionic Hadamard matrices. As a reminder, the Butson and Quaternionic Hadamard matrices are defined next.\\

A \emph{Butson Hadamard matrix} $H_n$ is a square matrix of order $n$, with entries from the set of complex $m^\text{th}$ roots of unity, such that $H_nH_n^* = nI_n$, where $(.)^*$ is the usual conjugate transpose. Such a matrix will be shown by BH$(n,m)$.\\

Let $\mathbb{H}$ be the non-commutative associative $\mathbb{R}$-algebra of quaternions:
$$
\mathbb{H}=\{q_1+q_2 i+q_3 j +q_4 k \mid q_\ell\in \mathbb{R}\quad (\ell\in\{1,2,3,4\})\}
$$
where $i^2=j^2=k^2=ijk=-1$.

A \emph{quaternionic Hadamard matrix of order $n$} is a square matrix $H_n$ of order $n$, with entries in the set $\{\pm1,\pm i, \pm j,\pm k\}$, such that $H_nH_n^* = nI_n$.

\subsection{Orthogonal designs}
There are many generalizations of Hadamard and weighing matrices. For the purposes at hand, we will focus on one such generalization, that is, orthogonal designs.
\begin{definition}
Let $x_1,x_2,\dots,x_u$ be real indeterminates. Then an \emph{orthogonal design} of order  $n$ and in variables $x_1,x_2,\dots,x_u$, denoted  OD$(n;s_1,s_2,\dots,s_u)$, is a square matrix $X_n$ of order $n$, with entries from the set $\{0,\pm x_1,\dots,\pm x_u\}$, such that $X_nX_n^t = \sigma I_n$, where $\sigma = \sum_{\ell=1}^u s_\ell x_\ell^2$. 
\end{definition} 
By replacing each indeterminate with unity, one obtains a weighing matrix of weight $\sum_{\ell=1}^u s_\ell$; and if $\sum_{\ell=1}^u s_\ell = n$, then one has a Hadamard matrix.

By changing the base field, the orthogonal designs may be extended to \emph{complex orthogonal designs} or \emph{quaternionic orthogonal designs} as follows.
\begin{definition}
Let $x_1,x_2,\dots,x_u$ be complex indeterminates. Then a \emph{complex orthogonal design} of order $n$ and in variables $x_1,x_2,\dots,x_u$, denoted  COD$(n;s_1,s_2,\dots,s_u)$, is a square matrix $X_n$ of order $n$, with entries from the set
$\{0,\pm\varepsilon_\ell x_\ell, \pm\varepsilon_\ell x_\ell^* \mid 1\leq \ell \leq u, \varepsilon_\ell \in \{1,i\}\}$
such that $XX^* = \sigma I_n$, where $\sigma = \sum_{\ell=1}^u s_\ell|x_\ell|^2$. 
\end{definition}
\begin{definition}
Let $x_1,x_2,\dots,x_u$ be complex indeterminates. A \emph{(restricted) quaternionic orthogonal design} is a square matrix $X_n$ of order $n$, with entries from the set $\{0,\pm\varepsilon_\ell x_\ell, \pm\varepsilon_\ell x_\ell^* \mid 1\leq \ell \leq u, \varepsilon_\ell \in \{1,i,j,k\}\}$ such that $X_nX_n^* = \sigma I_n$, and $\sigma = \sum_{\ell=1}^u s_\ell |x_\ell|^2$. One then writes $X_n$ is a QOD$(n;s_1,s_2,\dots,s_u)$.
\end{definition}
\begin{remark}
Note that if the coefficients of the variables of a QOD are in $\{\pm1,\pm i\}$, then it is a COD; and that if the coefficients of the variables in a COD are in $\{\pm1\}$, and if the variables are restricted to be real, then it is an OD.
\end{remark}

A (real, complex, or quaternionic) orthogonal design of order $n$  is said to be \emph{full} if $\sum_{\ell=1}^u s_{\ell}=n$.

Historically, there have been two main approaches to the subject, namely, the algebraic and the combinatorial methods. The former has successfully been completed for many cases by obtaining necessary conditions for the existence (see \cite{Seb}).


\subsection{Equiangular tight frames}
A \emph{frame} for $\mathbb{H}^d$ is a sequence of vectors $\{v_j\}_{j=1}^n$ such that there are some constants $A,B>0$ (frame constants) such that
$$
A||v||^2\leq \sum_{\ell=1}^n |\langle v,v_\ell\rangle|^2 \leq B||v||^2
$$
for any $v\in\mathbb{H}^d$.
A frame is said to be \emph{tight} if $A=B$ holds.
By \cite[Proposition 2.2]{waldron2020tight}, $V=\begin{bmatrix} v_1& v_2& \cdots & v_n\end{bmatrix}$ is a tight frame with the frame constant $A$ if and only if $P=\frac{1}{A}V^*V$ satisfies that $P^2=P$.

A tight frame $\{v_j \}_{j=1}^n$ is said to be \emph{equiangular} if $V=\begin{bmatrix} v_1& v_2& \cdots & v_n\end{bmatrix}$ satisfies that the off-diagonal entries of $V^*V$ have all the same absolute value.


\section{Balancedly splittable orthogonal designs}
As a generalization of the case of Hadamard matrices, we define balancedly splittable orthogonal designs as follows.
\begin{definition}\label{def:bsod}
Let $X_n$ be a full QOD$(n;s_1,\ldots,s_u)$.
The orthogonal design $X_n$ is said to be \emph{balancedly splittable} if $X_n$ contains an $m\times n$ submatrix $X_1$ where one of the following three conditions holds: for some $\alpha,\beta\in\mathbb{H}$,
\begin{enumerate}
\item the off diagonal entries of $X_1^*X_1$ are in the set
$$
\{\pm\varepsilon c x_1^{\ell_1}\cdots x_u^{\ell_u}(x_1^*)^{\ell'_1}\cdots (x_u^*)^{\ell'_u} \mid \ell_i,\ell_i'\in\mathbb{Z}_{\geq 0}, \varepsilon\in\{1,i,j,k\},c\in\{\alpha,\alpha^*,\beta,\beta^*\}\}
,$$
\item the off diagonal entries of $X_1^*X_1$ are in the set
\begin{align}\label{eq:diag}
\left\{\sum_{i=1}^u t_i |x_i|^2 \mid t_i\in\mathbb{Z}_{\geq 0},\sum_{i=1}^u t_i=m \right\}
\end{align}
or
$$
\{\pm\varepsilon c x_1^{\ell_1}\cdots x_u^{\ell_u}(x_1^*)^{\ell'_1}\cdots (x_u^*)^{\ell'_u} \mid \ell_i,\ell_i'\in\mathbb{Z}_{\geq 0}, \varepsilon\in\{1,i,j,k\},c\in\{\alpha,\alpha^*,\beta,\beta^*\}\}
$$
and there is some off-diagonal entry belonging to the set \eqref{eq:diag},
\item the off diagonal entries of $X_1^*X_1$ are in the set
$$
\{\pm\varepsilon c \sigma  \mid \varepsilon\in\{1, i, j, k\},c\in\{\alpha,\alpha^*,\beta,\beta^*\}\}
$$
where $\sigma = \sum_{\ell=1}^u s_\ell|x_\ell|^2$. 
\end{enumerate}
In the first case, we say that the split is \emph{unstable}; and in the second case, the split is \emph{unfaithful unstable}; while in the last, the split is \emph{stable}.
\end{definition}
We will use \emph{faithful} for both the first and third cases.

\begin{remark}
Note that the diagonal entries of $X_1^*X_1$ are in the set \eqref{eq:diag}.
\end{remark}

We shall see that a design may have both possibilities in (ii)  simultaneously. Moreover, the special case in which $|\alpha| = |\beta|$ will be of particular interest to us, at which point $X_1$ is  an equiangular tight frame.
\begin{proposition}\label{prop:ETF}
Let $X=\begin{bmatrix} X_1 \\ X_2\end{bmatrix}$ be a faithful balancedly splittable quartenionic orthogonal design QOD$(n;s_1,\ldots,s_u)$ with $\sum_{\ell=1}^us_\ell=n$ and $|\alpha|=|\beta|$. Then the column vectors of the $m\times n$ submatrix $X_1$ form an equiangular tight frame with frame constant $n$ that is parametrized by the complex variables $x_1,\ldots,x_u$, provided that $|x_1|=\cdots=|x_u|=1$.
\end{proposition}

\section{Construction: unfaithful case}
In this section, a construction method for real or complex orthogonal designs with the property of being unfaithful is shown.

First a construction for orthogonal designs OD$(q^m(q+1);q^m,q^{m+1})$.
Let $q$ be the order of a skew symmetric core $Q$ of a skew symmetric conference matrix $S$, that is, $Q$ is defined as
$$
DSD=\begin{bmatrix} 0 & {\bf 1}^\top \\ -\bf 1 & Q 
\end{bmatrix}
$$
where $D$ is some diagonal matrix with diagonal entries in $\{1,-1\}$ and $\bf 1$ denotes the all-ones vector.
 Define the following matrices recursively for each nonnegative integer $m$.
\begin{align}
\label{eqn:JA}
\Jq_m & = \begin{cases}
aJ_1 & \text{if } m=0,\\
J_q \otimes \mathcal{A}^{(q)}_{m-1} & \text{if }m\geq1,
\end{cases}     &
\Aq_m & = \begin{cases}
bJ_1 & \text{if } m=0,\\
I_q \otimes \mathcal{J}^{(q)}_{m-1} + Q \otimes \mathcal{A}^{(q)}_{m-1} & \text{if }m\geq1,
\end{cases}
\end{align}
where $a$ and $b$ are variables and $J_n$ is the $n\times n$ all-ones matrix. We write $\J_m=\J_m^{(q)}$ and $\A_m=\A_m^{(q)}$.
The following is similarly shown as \cite[Section~3]{FKS2018}. 
\begin{lemma}
The following hold.
\begin{enumerate}
\item $\J_{m} \A_{m}^t =  \A_{m}\J_{m}^t$.
\item $\J_m\J_m^t + q \A_m\A_m^t = (q^m a^2+q^{m+1}b^2)I_{q^m}$.
\item $\J_{1}^t \J_{1}=qa^2J_q,  \A_{1}^t \A_{1}=a^2 I_q +  b^2(qI_q-J_q), \A_{1}^t\J_{1}=\J_{1}^t\A_{1}=ab J_q$.
\end{enumerate}
\end{lemma}
We are ready for the case of real orthogonal designs.
\begin{theorem}
Let $X_m^{(q)}=I_{q+1}\otimes \Jq_m+S\otimes \Aq_m$.
\begin{enumerate}
\item The matrix $X_m^{(q)}$ is an orthogonal design OD$(q^m(q+1);q^m,q^{m+1})$.
\item The matrix $X_1^{(q)}$ is an unfaithful balancedly splittable
 orthogonal design OD$(q(q+1);q,q^2)$.
\end{enumerate}
\end{theorem}
\begin{proof}
Write $X_m=X_m^{(q)}$. The matrix $X_m$ has entries in $\{\pm a,\pm b\}$ and
\begin{align*}
X_m X_m^t
&=(I_{q+1}\otimes \J_m)(I_{q+1}\otimes \J_m)^t+(I_{q+1}\otimes \J_m)(S\otimes \A_m)^t\\
&\quad\quad+(S\otimes \A_m)(I_{q+1}\otimes \J_m)^t+(S\otimes \A_m)(S\otimes \A_m)^t\\
&=I_{q+1}\otimes \J_m {\J_m}^t+S^t\otimes \J_m\A_m^t+S\otimes \A_m\J_m^t+SS^t\otimes \A_m\A_m^t\\
&=I_{q+1}\otimes \J_m \J_m^t+(S+S^t)\otimes \J_m\A_m^t+q I_{q+1}\otimes \A_m\A_m^t\\
&=I_{q+1}\otimes (\J_m \J_m^t+q\A_m\A_m^t)\\
&=I_{q+1}\otimes(q^m a+q^{m+1}b)I_{q^m}\\
&=(q^m a^2+q^{m+1}b^2)I_{(q+1)q^m}.
\end{align*}
Therefore $X_m$ is an orthogonal design.

Take $X'=\begin{bmatrix}\J_1& \A_1 & \cdots & \A_1\end{bmatrix}$ as the submatrix of $X_1$ whose rows are in the first $q$ rows of $X_1^{(q)}$.
Then we calculate $(X')^t X'$ as follows:
\begin{align*}
(X')^t X' &=\begin{bmatrix}
\J_1^t\J_1 & \J_1^t\A_1 & \cdots & \J_1^t\A_1\\
\A_1^t\J_1 & \A_1^t\A_1 & \cdots & \A_1^t\A_1\\
\vdots & \vdots & \ddots & \vdots  \\
\A_1^t\J_1 & \A_1^t\A_1 & \cdots & \A_1^t\A_1
\end{bmatrix}\\
&=\begin{bmatrix}
qa^2J_q & ab J_q & \cdots & ab J_q \\
ab J_q & a^2 I_q +  b^2(qI_q-J_q) & \cdots & a^2 I_q +  b^2(qI_q-J_q) \\
\vdots & \vdots & \ddots & \vdots  \\
ab J_q & a^2 I_q +  b^2(qI_q-J_q) & \cdots & a^2 I_q +  b^2(qI_q-J_q)
\end{bmatrix}.
\end{align*}
Thus, $X_1$ is an unfaithful balancedly splittable orthogonal design.
\end{proof}

Next, we construct a complex orthogonal design COD$(q^m(q+1);q^m,q^{m+1})$.

Let $q$ be the order of a symmetric core $Q$ of a symmetric conference matrix $S$. Define the following matrices recursively for each nonnegative integer $m$.
\begin{align}
\label{eqn:CD}
\Cq_m & = \begin{cases}
aJ_1 & \text{if } m=0,\\
J_q \otimes \Dq_{m-1} & \text{if }m\geq1,
\end{cases}     &
\Dq_m & = \begin{cases}
bJ_1 & \text{if } m=0,\\
I_q \otimes \Cq_{m-1} + i \, Q \otimes \Dq_{m-1} & \text{if }m\geq1,
\end{cases}
\end{align}
where $a$ and $b$ are variables. We write $\C_m=\C_m^{(q)}$ and $\D_m=\D_m^{(q)}$.
Then it is similarly shown as the case $\J_m,\A_m$ that:
\begin{lemma}
The following hold.
\begin{enumerate}
\item $\C_{m} \D_{m}^* =  \D_{m}\C_{m}^*$.
\item $\C_m\C_m^* + q \D_m\D_m^* = (q^m a^2+q^{m+1}b^2)I_{q^m}$.
\item $\C_{1}^* \D_{1}=qa^2J_q,  \D_{1}^* \D_{1}=a^2 I_q +  b^2(qI_q-J_q), \D_{1}^*\C_{1}=\C_{1}^*\D_{1}=ab J_q$.
\end{enumerate}
\end{lemma}
\begin{theorem}
Let $Y_m^{(q)}=i I_{q+1}\otimes \Cq_m+S\otimes \Dq_m$.
\begin{enumerate}
\item The matrix $Y_m^{(q)}$ is a complex orthogonal design COD$(q^m(q+1);q^m,q^{m+1})$.
\item $Y_1^{(q)}$ is an unfaithful balancedly splittable complex orthogonal design.
\end{enumerate}
\end{theorem}
\begin{proof}
Write $Y_m=Y_m^{(q)}$.
The matrix $Y_m$ has entries in $\{\pm a,\pm i b\}$ and
\begin{align*}
Y_m Y_m^*
&=I_{q+1}\otimes \C_m \C_m^*+iS^*\otimes \C_m\D_m^*-iS\otimes \D_m\C_m^*+SS^*\otimes \D_m\D_m^*\\
&=I_{q+1}\otimes (\C_m \C_m^*+q\D_m\D_m^*)\\
&=I_{q+1}\otimes(q^m a+q^{m+1}b)I_{q^m}\\
&=(q^m a^2+q^{m+1}b^2)I_{(q+1)q^m}.
\end{align*}
Therefore $Y_m$ is a complex orthogonal design.

Take $Y'=\begin{bmatrix}i\C_1& \D_1 & \cdots & \D_1\end{bmatrix}$ as the submatrix of $Y_1$ whose rows are in the first $q$ rows of $Y_1^{(q)}$.
Then we calculate $(Y')^t Y'$ as follows:
\begin{align*}
(Y')^* Y' &=\begin{bmatrix}
\C_1^*\C_1 & -i\C_1^*\D_1 & \cdots & -i\C_1^*\D_1\\
i\D_1^*\C_1 & \D_1^*\D_1 & \cdots & \D_1^*\D_1\\
\vdots & \vdots & \ddots & \vdots  \\
i\D_1^*\C_1 & \D_1^*\D_1 & \cdots & \D_1^*\D_1
\end{bmatrix}\\
&=\begin{bmatrix}
qa^2J_q & -iab J_q & \cdots & -iab J_q \\
iab J_q & a^2 I_q +  b^2(qI_q-J_q) & \cdots & a^2 I_q +  b^2(qI_q-J_q) \\
\vdots & \vdots & \ddots & \vdots  \\
iab J_q & a^2 I_q +  b^2(qI_q-J_q) & \cdots & a^2 I_q +  b^2(qI_q-J_q)
\end{bmatrix}.
\end{align*}
Thus, $Y_1$ is an unfaithful balancedly splittable complex orthogonal design.
\end{proof}

\section{Construction: faithful case}
\subsection{A recursive method}
\indent
It is the field's characteristic to use existing designs by replacing the variables with suitable plug-in matrices. By far, the most successful plug-in matrices are the circulant-type matrices. We will again find these helpful in our construction, and so we need the following modification of Barker sequences and Golay pairs.

A \emph{Barker sequence} is a sequence $C=(C_i)_{i=0}^{n-1} \subseteq \mathbb{H}^n$ such that the aperiodic autocorrelations satisfy $N_C(k) = \sum_{i=0}^{n-k-1}C_{i+k}C_i^* = 0$, for all $0 \leq k < n-1$. Note that we do not distinguish between this definition and the case when the $C_i$ are matrices over $\mathbb{H}$ satisfying the same conditions. We need one further idea.

Let $A = (A_i)_{i=0}^{n-1}$ and $B = (B_i)_{i=0}^{n-1}$ be two sequences such that $N_A(k) + N_B(k) = 0$, for every $0 \leq k < n-1$; then $A$ and $B$ are said to be a \emph{complementary Golay pair}. If $(A,B)$ denotes the concatination of the sequences $A$ and $B$, then it can be shown that $(A,B)$ and $(A,-B)$ are also a Golay pair, where $-B=(-B_i)_{i=0}^{n-1}$. If $C = (C_i)_{i=0}^{n-1}$ is some Barker sequence, then  $(C_0,C_1,\dots,C_{n-1},C_{n-1}, \dots, C_{1})$ and $(C_0,C_1,\dots,C_{n-1},-C_{n-1},\dots,-C_1)$ is a Golay pair. We are now ready to proceed with the construction.

Let $X_n$ be a QOD$(n;s_1,s_2,\dots,s_u)$, such that $\sum_{i=1}^u s_i = n$, and let $H_n$ be a quaternionic Hadamard matrix of order $n$.
Furthermore, take $H_{2n} = \left[\begin{smallmatrix} 1 & 1 \\ 1 & - \end{smallmatrix}\right] \otimes H_n$ and $X_{2n} = \left[\begin{smallmatrix} 1 & 1 \\ 1 & - \end{smallmatrix}\right] \otimes X_n$; and index the rows of $H_n$, $H_{2n}$, $X_n$, and $X_{2n}$ by $h_i$, $\hat{h}_j$, $r_i$, and $\hat{r}_j$, respectively, for $0 \leq i< n,0\leq j<2n$.

The so-called \emph{auxiliary matrices} $c_i = h_i^* h_i$ of $H_n$ were introduced in \cite{Khar} for the case of a real Hadamard matrix $H_n$ and satisfy the following.

\begin{lemma}
Let $H_n$ be a quaternion Hadamard matrix of order $n$, and let $c_i = h_i^* h_i$, where $h_i$ is the i$^\text{th}$ row of $H_n$. Then:
\begin{enumerate}
\item $\sum_{i=0}^{n-1} c_i = nI_n$, and
\item $c_ic_j^* = 0$ whenever $i \neq j$.
\end{enumerate}
\end{lemma}

\indent Following \cite{KharSuda1}, we may extend this result to include orthogonal designs by defining $C_i = h_i^* r_i$. 
Recall that $\sigma = \sum_{\ell=1}^u s_\ell x_\ell x_\ell^*$. 
\begin{lemma}
\begin{enumerate}
\item $\sum_{i=0}^{n-1} C_iC_i^* = \sigma nI_n$,
\item $C_iC_i^* = \sigma c_i$ for any $i$, and
\item $C_iC_j^* = 0$ whenever $i \neq j$.
\end{enumerate}
\end{lemma}
Moreover, the sequence $(C_i)_{i=0}^{n-1}$ is a Barker sequence, so  the sequences $(C_0,C_1,\ldots,C_{n-1},$ $C_{n-1},\dots,C_1)$ and $(C_0,C_1,\dots,C_{n-1},-C_{n-1},\dots,-C_1)$ form a Golay pair. \\
\indent Define
\begin{align*}
A &= \text{circ}(C_0,C_1,\dots,C_{n-1},C_{n-1},\dots,C_1) \text{, and} \\
B &= \text{circ}(C_0,C_1,\dots,C_{n-1},-C_{n-1},\dots,-C_1).
\end{align*}
Take $E^* = [E_1^* \cdots E_{2n-1}^*]^*$ and $F = [F_1 \cdots F_{2n-1}]$, where $E_i = h_0^*\hat{r}_i$ and $F_i = \hat{h}_i^* r_0$. Finally, take $G = \hat{h}_0^*\hat{r}_0$, and define
\begin{equation}
X_{4n^2} =
\begin{bmatrix}\label{eq:X}
G & F & -F \\
E & A & B \\
-E & B & A
\end{bmatrix}.
\end{equation}
As a first result, we have the following.

\begin{theorem}\label{thm:qod}
Let $X_{4n^2}$ be as in \eqref{eq:X}. Then $X_{4n^2}$ is a QOD$(4n^2; 4ns_1,4ns_2,\dots,4ns_u)$.
\end{theorem}
\begin{proof}
The proof follows by a tedious checking of the block entries of $X_{4n^2}X_{4n^2}^*$.

First, we show that the off-diagonal blocks of the product resolve to zero. We have that
\begin{align*}
GE_i^* = (\hat{h}_0^*\hat{r}_0)(h_0^*\hat{r}_i)^* = \hat{h}_0^*(\hat{r}_0\hat{r}_i^*)h_0 = 0,
\end{align*}
 so that $GE^* = 0$. Similarly, $EG^* = 0$. Then $F_iC_j^* = (\hat{h}_i^* r_0)(h_j^* r_j)^* = \hat{h}_i^*(r_0r_j^*)h_j = \delta_{j,0}\sigma\hat{h}_i^* h_j = \delta_{j,0}\sigma\hat{h}_i^* h_0$; whence,
 \begin{align}\label{eq:fa}
 FA^* = FB^* = \sigma F \otimes \hat{h}_i^* h_0
 \end{align}
 so that $FA^* - FB^* = AF^* - BF^* = 0$.

Now, we show that the product between the second and third block rows resolves to zero.
Observe that
\begin{align*}
E_iE_j^* = (h_0^*\hat{r}_i)(h_0^*\hat{r}_j)^* = h_0^*(\hat{r}_i\hat{r}_j^*)h_0= 2\sigma \delta_{i,j}  h_0^* h_0. 
\end{align*}
Hence, $EE^* = 2\sigma I_{2n-1} \otimes h_0^* h_0$. Next, we have that
\begin{align}
AB^* &= \sigma\text{circ}(h_0^* h_0, c_{n-1},-c_1,c_{n-2},\dots,c_{n-1}),\label{eq:ab}\\
BA^* &= \sigma\text{circ}(h_0^* h_0,-c_{n-1},c_1,-c_{n-2},\dots,-c_{n-1}),\label{eq:ba}
\end{align}
 so that $AB^* + BA^* = 2\sigma I_{2n-1} \otimes  h_0^* h_0$.
Thus, $-EE^* + AB^* + BA^* = 0$.

It remains to consider the block diagonal entries of the product. We have $GG^* = 2\sigma \hat{h}_0^* \hat{h}_0$ and
\begin{align*}
FF^* = \sum_{i=1}^{2n-1} F_iF_i^* = \sigma\left(\sum_{i=0}^{2n-1} \hat{h}_i^*\hat{h}_i - \hat{h}_0^*\hat{h}_0\right) = \sigma(2n I_{2n} -\hat{h}_0^*\hat{h}_0);
\end{align*}
 whence, $GG^* + 2FF^* = 4n\sigma I_{2n}$. Next, since $A$ and $B$ are complementary, it follows that
\begin{align*}
AA^* + BB^* = I_{2n-1} \otimes 2(C_0C_0^t + 2 \sum_{i=1}^{n-1}C_iC_i^*) = 2\sigma I_{2n-1} \otimes (2nI_n - h_0^* h_0).
\end{align*}
Therefore, $EE^* + AA^* + BB^* = 4n\sigma I_{2n^2-n}$.

This completes the proof  that $X_{4n^2}$ is a QOD$(4n^2;4ns_1,4ns_2,\dots,4ns_u)$.
\end{proof}

\indent We now illustrate the construction with several examples.
In the following examples, $-$ stands for $-1$, the overline of variable denotes the negation and upper asterisk denotes the conjugate of the variable.

To explain the above construction, we show all the steps one by one in following example.

\begin{example}

We begin with the OD$(2;1,1)$ given by $X_2 = [\begin{smallmatrix} a & b \\ b & \bar{a} \end{smallmatrix}]$, where $a$ and $b$ are commuting real indeterminants. If $H_2 = [\begin{smallmatrix} 1 & 1 \\ 1 & - \end{smallmatrix}]$, then form $H_4 = H_2 \otimes H_2$ and $X_4 = H_2 \otimes X_2$. Forming the auxiliary matrices of $X_2$, we then construct $A$ and $B$ as in the Theorem.
$$
\arraycolsep=1.0pt\def\arraystretch{1.0}
A=
\left[
\begin{array}{cc|cc|cc}
  a&  b&  b&  \bar{a}&  b&  \bar{a}\\
  a&  b&  \bar{b}&  a&  \bar{b}&  a\\ \hline
  b&  \bar{a}&  a&  b&  b&  \bar{a}\\
  \bar{b}&  a&  a&  b&  \bar{b}&  a\\ \hline
  b&  \bar{a}&  b&  \bar{a}&  a&  b\\
  \bar{b}&  a&  \bar{b}&  a&  a&  b
\end{array}
\right]
\qquad
B=
\left[
\begin{array}{cc|cc|cc}
  a&  b&  b&  \bar{a}&  \bar{b}&  a\\
  a&  b&  \bar{b}&  a&  b&  \bar{a}\\ \hline
  \bar{b}&  a&  a&  b&  b&  \bar{a}\\
  b&  \bar{a}&  a&  b&  \bar{b}&  a\\ \hline
  b&  \bar{a}&  \bar{b}&  a&  a&  b\\
  \bar{b}&  a&  b&  \bar{a}&  a&  b
\end{array}
\right]
$$
In a similar way, we construct the remaining block matrices.
$$
\arraycolsep=1.0pt\def\arraystretch{1.0}
G=
\left[
\begin{array}{cccc}
  a&  b&  a&  b\\
  a&  b&  a&  b\\
  a&  b&  a&  b\\
  a&  b&  a&  b
\end{array}
\right]
\qquad
E=
\left[
\begin{array}{cccc}
  b&  \bar{a}&  b&  \bar{a}\\
  b&  \bar{a}&  b&  \bar{a}\\ \hline
  a&  b&  \bar{a}&  \bar{b}\\
  a&  b&  \bar{a}&  \bar{b}\\ \hline
  b&  \bar{a}&  \bar{b}&  a\\
  b&  \bar{a}&  \bar{b}&  a
\end{array}
\right]
\qquad
F=
\left[
\begin{array}{cc|cc|cc}
  a&  b&  a&  b&  a&  b\\
  \bar{a}&  \bar{b}&  a&  b&  \bar{a}&  \bar{b}\\
  a&  b&  \bar{a}&  \bar{b}&  \bar{a}&  \bar{b}\\
  \bar{a}&  \bar{b}&  \bar{a}&  \bar{b}&  a&  b
\end{array}
\right]
$$
Putting these together as in \eqref{eq:X}, we obtain a balancedly splittable OD$(16;8,8)$.
$$
\arraycolsep=1.0pt\def\arraystretch{1.0}
\left[
\begin{array}{cccc|cccccc|cccccc}
  a&  b&  a&  b&  a&  b&  a&  b&  a&  b&  \bar{a}&  \bar{b}&  \bar{a}&  \bar{b}&  \bar{a}&  \bar{b}\\
  a&  b&  a&  b&  \bar{a}&  \bar{b}&  a&  b&  \bar{a}&  \bar{b}&  a&  b&  \bar{a}&  \bar{b}&  a&  b\\
  a&  b&  a&  b&  a&  b&  \bar{a}&  \bar{b}&  \bar{a}&  \bar{b}&  \bar{a}&  \bar{b}&  a&  b&  a&  b\\
  a&  b&  a&  b&  \bar{a}&  \bar{b}&  \bar{a}&  \bar{b}&  a&  b&  a&  b&  a&  b&  \bar{a}&  \bar{b}\\ \hline
  b&  \bar{a}&  b&  \bar{a}&  a&  b&  b&  \bar{a}&  b&  \bar{a}&  a&  b&  b&  \bar{a}&  \bar{b}&  a\\
  b&  \bar{a}&  b&  \bar{a}&  a&  b&  \bar{b}&  a&  \bar{b}&  a&  a&  b&  \bar{b}&  a&  b&  \bar{a}\\
  a&  b&  \bar{a}&  \bar{b}&  b&  \bar{a}&  a&  b&  b&  \bar{a}&  \bar{b}&  a&  a&  b&  b&  \bar{a}\\
  a&  b&  \bar{a}&  \bar{b}&  \bar{b}&  a&  a&  b&  \bar{b}&  a&  b&  \bar{a}&  a&  b&  \bar{b}&  a\\
  b&  \bar{a}&  \bar{b}&  a&  b&  \bar{a}&  b&  \bar{a}&  a&  b&  b&  \bar{a}&  \bar{b}&  a&  a&  b\\
  b&  \bar{a}&  \bar{b}&  a&  \bar{b}&  a&  \bar{b}&  a&  a&  b&  \bar{b}&  a&  b&  \bar{a}&  a&  b\\ \hline
  \bar{b}&  a&  \bar{b}&  a&  a&  b&  b&  \bar{a}&  \bar{b}&  a&  a&  b&  b&  \bar{a}&  b&  \bar{a}\\
  \bar{b}&  a&  \bar{b}&  a&  a&  b&  \bar{b}&  a&  b&  \bar{a}&  a&  b&  \bar{b}&  a&  \bar{b}&  a\\
  \bar{a}&  \bar{b}&  a&  b&  \bar{b}&  a&  a&  b&  b&  \bar{a}&  b&  \bar{a}&  a&  b&  b&  \bar{a}\\
  \bar{a}&  \bar{b}&  a&  b&  b&  \bar{a}&  a&  b&  \bar{b}&  a&  \bar{b}&  a&  a&  b&  \bar{b}&  a\\
  \bar{b}&  a&  b&  \bar{a}&  b&  \bar{a}&  \bar{b}&  a&  a&  b&  b&  \bar{a}&  b&  \bar{a}&  a&  b\\
  \bar{b}&  a&  b&  \bar{a}&  \bar{b}&  a&  b&  \bar{a}&  a&  b&  \bar{b}&  a&  \bar{b}&  a&  a&  b
\end{array}
\right]
$$

\end{example}

\begin{example}
Consider the COD$(2;1,1)$ given by
$$
\begin{bmatrix}
a & b \\
\bar{b}^* & a^*
\end{bmatrix}.
$$
Using the construction, we have the following COD$(16;8,8)$
$$
\arraycolsep=1.0pt\def\arraystretch{1.0} 
\left[
\begin{array}{cccc|cccccc|cccccc}
 a & b & a & b & a & b & a & b & a & b & \bar{a} & \bar{b} & \bar{a} & \bar{b} & \bar{a} & \bar{b} \\
 a & b & a & b & \bar{a} & \bar{b} & a & b & \bar{a} & \bar{b} & a & b & \bar{a} & \bar{b} & a & b \\
 a & b & a & b & a & b & \bar{a} & \bar{b} & \bar{a} & \bar{b} & \bar{a} & \bar{b} & a & b & a & b \\
 a & b & a & b & \bar{a} & \bar{b} & \bar{a} & \bar{b} & a & b & a & b & a & b & \bar{a} & \bar{b} \\ \hline
 \bar{b}^* & a^* & \bar{b}^* & a^* & a & b & \bar{b}^* & a^* & \bar{b}^* & a^* & a & b & \bar{b}^* & a^* & b^* & \bar{a}^* \\
 \bar{b}^* & a^* & \bar{b}^* & a^* & a & b & b^* & \bar{a}^* & b^* & \bar{a}^* & a & b & b^* & \bar{a}^* & \bar{b}^* & a^* \\
 a & b & \bar{a} & \bar{b} & \bar{b}^* & a^* & a & b & \bar{b}^* & a^* & b^* & \bar{a}^* & a & b & \bar{b}^* & a^* \\
 a & b & \bar{a} & \bar{b} & b^* & \bar{a}^* & a & b & b^* & \bar{a}^* & \bar{b}^* & a^* & a & b & b^* & \bar{a}^* \\
 \bar{b}^* & a^* & b^* & \bar{a}^* & \bar{b}^* & a^* & \bar{b}^* & a^* & a & b & \bar{b}^* & a^* & b^* & \bar{a}^* & a & b \\
 \bar{b}^* & a^* & b^* & \bar{a}^* & b^* & \bar{a}^* & b^* & \bar{a}^* & a & b & b^* & \bar{a}^* & \bar{b}^* & a^* & a & b \\ \hline
 b^* & \bar{a}^* & b^* & \bar{a}^* & a & b & \bar{b}^* & a^* & b^* & \bar{a}^* & a & b & \bar{b}^* & a^* & \bar{b}^* & a^* \\
 b^* & \bar{a}^* & b^* & \bar{a}^* & a & b & b^* & \bar{a}^* & \bar{b}^* & a^* & a & b & b^* & \bar{a}^* & b^* & \bar{a}^* \\
 \bar{a} & \bar{b} & a & b & b^* & \bar{a}^* & a & b & \bar{b}^* & a^* & \bar{b}^* & a^* & a & b & \bar{b}^* & a^* \\
 \bar{a} & \bar{b} & a & b & \bar{b}^* & a^* & a & b & b^* & \bar{a}^* & b^* & \bar{a}^* & a & b & b^* & \bar{a}^* \\
 b^* & \bar{a}^* & \bar{b}^* & a^* & \bar{b}^* & a^* & b^* & \bar{a}^* & a & b & \bar{b}^* & a^* & \bar{b}^* & a^* & a & b \\
 b^* & \bar{a}^* & \bar{b}^* & a^* & b^* & \bar{a}^* & \bar{b}^* & a^* & a & b & b^* & \bar{a}^* & b^* & \bar{a}^* & a & b \\
\end{array}
\right].
$$
\end{example}

\begin{example}
The following is a QOD$(2;1,1)$
$$
\begin{bmatrix}
\bar{a} & bi \\
\bar{b}j & ak
\end{bmatrix},
$$
where $a$ and $b$ are real variables. We then arrive at the following QOD$(16;8,8)$
\[
\arraycolsep=1.0pt\def\arraystretch{1.0}
\left[
\begin{array}{cccc|cccccc|cccccc}
\bar{a}&bi&\bar{a}&bi&\bar{a}&bi&\bar{a}&bi&\bar{a}&bi& a&\bar{b}i& a&\bar{b}i& a&\bar{b}i\\
\bar{a}&bi&\bar{a}&bi& a&\bar{b}i&\bar{a}&bi& a&\bar{b}i&\bar{a}&bi& a&\bar{b}i&\bar{a}&bi\\
\bar{a}&bi&\bar{a}&bi&\bar{a}&bi& a&\bar{b}i& a&\bar{b}i& a&\bar{b}i&\bar{a}&bi&\bar{a}&bi\\
\bar{a}&bi&\bar{a}&bi& a&\bar{b}i& a&\bar{b}i&\bar{a}&bi&\bar{a}&bi&\bar{a}&bi& a&\bar{b}i\\\hline
\bar{b}j&ak&\bar{b}j&ak&\bar{a}&bi&\bar{b}j&ak&\bar{b}j&ak&\bar{a}&bi&\bar{b}j&ak&bj&\bar{a}k\\
\bar{b}j&ak&\bar{b}j&ak&\bar{a}&bi&bj&\bar{a}k&bj&\bar{a}k&\bar{a}&bi&bj&\bar{a}k&\bar{b}j&ak\\
\bar{a}&bi& a&\bar{b}i&\bar{b}j&ak&\bar{a}&bi&\bar{b}j&ak&bj&\bar{a}k&\bar{a}&bi&\bar{b}j&ak\\
\bar{a}&bi& a&\bar{b}i&bj&\bar{a}k&\bar{a}&bi&bj&\bar{a}k&\bar{b}j&ak&\bar{a}&bi&bj&\bar{a}k\\
\bar{b}j&ak&bj&\bar{a}k&\bar{b}j&ak&\bar{b}j&ak&\bar{a}&bi&\bar{b}j&ak&bj&\bar{a}k&\bar{a}&bi\\
\bar{b}j&ak&bj&\bar{a}k&bj&\bar{a}k&bj&\bar{a}k&\bar{a}&bi&bj&\bar{a}k&\bar{b}j&ak&\bar{a}&bi\\\hline
bj&\bar{a}k&bj&\bar{a}k&\bar{a}&bi&\bar{b}j&ak&bj&\bar{a}k&\bar{a}&bi&\bar{b}j&ak&\bar{b}j&ak\\
bj&\bar{a}k&bj&\bar{a}k&\bar{a}&bi&bj&\bar{a}k&\bar{b}j&ak&\bar{a}&bi&bj&\bar{a}k&bj&\bar{a}k\\
 a&\bar{b}i&\bar{a}&bi&bj&\bar{a}k&\bar{a}&bi&\bar{b}j&ak&\bar{b}j&ak&\bar{a}&bi&\bar{b}j&ak\\
 a&\bar{b}i&\bar{a}&bi&\bar{b}j&ak&\bar{a}&bi&bj&\bar{a}k&bj&\bar{a}k&\bar{a}&bi&bj&\bar{a}k\\
bj&\bar{a}k&\bar{b}j&ak&\bar{b}j&ak&bj&\bar{a}k&\bar{a}&bi&\bar{b}j&ak&\bar{b}j&ak&\bar{a}&bi\\
bj&\bar{a}k&\bar{b}j&ak&bj&\bar{a}k&\bar{b}j&ak&\bar{a}&bi&bj&\bar{a}k&bj&\bar{a}k&\bar{a}&bi\\
\end{array}
\right].
\]

\end{example}

Two quaternion matrices $A$ and $B$ are said to be \emph{amicable} if $AB^*=BA^*$. If $A$ and $B$ are amicable quaternionic orthogonal designs of order $n$ and types $(s_1,s_2,\dots,s_u)$ and $(t_1,t_2,\dots,t_v)$, respectively, then we say that $(A,B)$ is an \emph{amicable quarternionic orthogonal design}, and we write $(A,B)$ is an \\AQOD$(n;(s_1,s_2,\dots,$ $s_u);(t_1,t_2,\dots,t_v))$. It is an interesting  consequence of the construction that amicability is preserved.

\begin{corollary}
If $(X_n,Y_n)$ is an AQOD$(n;(s_1,s_2,\dots,s_u);(t_1,t_2,\dots,t_v))$, and if $X_{4n^2}$ and $Y_{4n^2}$ are given by \eqref{eq:X}, then $(X_{4n^2},Y_{4n^2})$ is an \\ AQOD$(4n^2;(4ns_1,4ns_2,\dots,4ns_u);(4nt_1,4nt_2,\dots,4nt_v))$.
\end{corollary}

\subsection{Balancedly splitted matrices  }
We next show that the constructed designs are balancedly splittable.

\begin{theorem}\label{thm:bqod}
Let $X_{4n^2}$ be as in \eqref{eq:X}, then it is balancedly splittable. In particular:
\begin{enumerate}
\item The submatrices $\left[ \begin{smallmatrix} E & A & B \end{smallmatrix} \right]$ and $\left[ \begin{smallmatrix} -E & B & A \end{smallmatrix} \right]$ induce an unstable split with $\alpha = -\beta = n$.
\item The submatrices $\left[ \begin{smallmatrix} F^* & A^* & B^* \end{smallmatrix} \right]^*$ and $\left[ \begin{smallmatrix} -F^* & B^* & A^* \end{smallmatrix} \right]^*$ induce a stable split with $\alpha = -\beta = 1$.
\end{enumerate}
\end{theorem}
\begin{proof}
The proof proceeds as in Theorem~\ref{thm:qod}. Note that it suffices to prove the result for $H = \left[ \begin{smallmatrix} E & A & B \end{smallmatrix} \right]$ and $V^* = \left[ \begin{smallmatrix} F^* & A^* & B^* \end{smallmatrix} \right]^*$.

Towards proving (i), define $$S = \{\pm\varepsilon n x_\ell x_{\ell'}, \pm\varepsilon n x_\ell x_{\ell'}^*, \pm\varepsilon n x_{\ell}^*x_{\ell'}^* \mid 1 \leq \ell,\ell' \leq u,\varepsilon\in\{1, i, j, k\}\}.$$
First note that
\begin{align}\label{eq:hh}
H^*H=\begin{bmatrix}
E^*E & E^* A & E^*B \\
A^*E & A^* A & A^*B \\
B^*E & B^* A & B^*B
\end{bmatrix}.
\end{align}

We have that
\begin{align}\label{eq:ee}
E^* E = \sum_{i=1}^{2n-1} (h_0^* \hat{r}_i)^*(h_0^t\hat{r}_i) = n\sum_{i=1}^{2n-1}\hat{r}_i^*\hat{r}_i = n(2\sigma I_{2n} - \hat{r}_0^*\hat{r}_0).
\end{align}
 The off-diagonal entries of $E^*E$ are then from the set $S$, which is sufficient. Now,
\begin{align*}
E_i^* C_j = (h_0^*\hat{r}_i)^*(h_j^* r_j) = n \delta_{j,0}\hat{r}_i^* r_j;
\end{align*}
whence, the entries of $E^* A$, $A^* E$, $E^* B$, and $B^* E$ are in $S$.
Further,
\begin{align}\label{eq:aa}
A^*A = n\text{circ}(2\sigma I_n - r_0^* r_0,r_{n-1}^* r_{n-1},r_1^* r_1,r_{n-2}^* r_{n-2},\dots,r_{n-1}^* r_{n-1}),
\end{align}
 and $B^* B$ is similar.
Thus, the off-diagonal entries of $A^* A$ and $B^* B$ are also in $S$.
The cases $A^*B$ and $B^*A$ are handled analogously.
It now follows that $H^* H$ has all off-diagonal entries in $S$. This completes part (i).

Over the course of proving Theorem~\ref{thm:qod}, it follows that
$$
VV^*=\begin{bmatrix}
\overline{F}F^t & \overline{F}A^t & \overline{F}B^t \\
\overline{A}F^t & \overline{A}A^t & \overline{A}B^t \\
\overline{B}F^t & \overline{B}A^t & \overline{B}B^t
\end{bmatrix}
=\begin{bmatrix}
\sigma(2n I_{2n} -J_{2n})
 & \sigma F \otimes \hat{h}_i^t h_0 & \sigma F \otimes \hat{h}_i^t h_0 \\
\sigma (F \otimes \hat{h}_i^t h_0)^t & \overline{A}A^t & \overline{A}B^t \\
\sigma (F \otimes \hat{h}_i^t h_0)^t & \overline{B}A^t & \overline{B}B^t
\end{bmatrix},
$$
where letting $P=\text{circ}(010\cdots0)$,
\begin{align*}
\overline{A}A^*&=\sigma(2nI_{(2n-1)n}-I_{2n-1}\otimes J_n+\sum_{i=1}^{n-1}(P^{2i}+P^{-2i})\otimes \overline{c_i}),\\
\overline{B}B^*&=\sigma(2nI_{(2n-1)n}-I_{2n-1}\otimes J_n-\sum_{i=1}^{n-1}(P^{2i}+P^{-2i})\otimes \overline{c_i}).
\end{align*}
In view of \eqref{eq:ab} and \eqref{eq:ba}, the off-diagonal entries of $VV^*$ are shown to be $\pm\sigma$. Part (ii), therefore, has been shown to be true.
\end{proof}
\begin{remark}
\begin{enumerate}
\item If the complex variables $x_j$ run over the set $\{c\in\mathbb{C} \mid |c|=1\}$, then the matrix $(1/(2n^2-n))H^*H$ is the Gram matrix of a quaternion equiangular tight frame whose off diagonal entries belong to the set
$$
\left\{\frac{\pm\varepsilon x_\ell x_{\ell'}}{2n-1},\frac{\pm\varepsilon x_\ell x_{\ell'}^*}{2n-1},\frac{\pm\varepsilon x_{\ell}^*x_{\ell'}^*}{2n-1} \mid 1 \leq \ell,\ell' \leq u,\varepsilon\in\{1,i,j,k\}\right\}.
$$
Thus, we obtain infinitely many quaternion ETFs parameterized by $u$ variables.
\item If the quaternion variables $x_j$ run over the set $\{h\in\mathbb{H} \mid |h|=1\}$, then matrix $(1/(2n^2-n))VV^*$ is the Gram matrix of a quaternion equiangular tight frame whose off diagonal entries belong to the set
$$
\left\{\pm\frac{\varepsilon \sigma}{2n-1} \mid \varepsilon\in\{1,i,j,k\}\right\}.
$$
Thus, we obtain infinitely many quaternion ETFs parameterized by $u$ variables, but all isomorphic.
\end{enumerate}
\end{remark}

\begin{example}

Continuing the above example for the splittable QOD(16;8,8), and using the first vertical frame $V^* = \left[\begin{smallmatrix} F^* & A^* & B^* \end{smallmatrix}\right]^*$, we have that $(1/\sigma)VV^*$ is given by \\
\[
\arraycolsep=1.0pt\def\arraystretch{1.0}
\left[
\begin{array}{cccccccccccccccc}
3 & - & - & - & 1 & 1 & 1 & 1 & 1 & 1 & 1 & 1 & 1 & 1 & 1 & 1 \\
- & 3 & - & - & - & - & 1 & 1 & - & - & - & - & 1 & 1 & - & - \\
- & - & 3 & - & 1 & 1 & - & - & - & - & 1 & 1 & - & - & - & - \\
- & - & - & 3 & - & - & - & - & 1 & 1 & - & - & - & - & 1 & 1 \\
1 & - & 1 & - & 3 & - & 1 & - & 1 & - & 1 & 1 & 1 & - & - & 1 \\
1 & - & 1 & - & - & 3 & - & 1 & - & 1 & 1 & 1 & - & 1 & 1 & - \\
1 & 1 & - & - & 1 & - & 3 & - & 1 & - & - & 1 & 1 & 1 & 1 & - \\
1 & 1 & - & - & - & 1 & - & 3 & - & 1 & 1 & - & 1 & 1 & - & 1 \\
1 & - & - & 1 & 1 & - & 1 & - & 3 & - & 1 & - & - & 1 & 1 & 1 \\
1 & - & - & 1 & - & 1 & - & 1 & - & 3 & - & 1 & 1 & - & 1 & 1 \\
1 & - & 1 & - & 1 & 1 & - & 1 & 1 & - & 3 & - & - & 1 & - & 1 \\
1 & - & 1 & - & 1 & 1 & 1 & - & - & 1 & - & 3 & 1 & - & 1 & - \\
1 & 1 & - & - & 1 & - & 1 & 1 & - & 1 & - & 1 & 3 & - & - & 1 \\
1 & 1 & - & - & - & 1 & 1 & 1 & 1 & - & 1 & - & - & 3 & 1 & - \\
1 & - & - & 1 & - & 1 & 1 & - & 1 & 1 & - & 1 & - & 1 & 3 & - \\
1 & - & - & 1 & 1 & - & - & 1 & 1 & 1 & 1 & - & 1 & - & - & 3
\end{array}
\right].
\]
Similarly, using the first horizontal frame $H = \left[\begin{smallmatrix} E & A & B \end{smallmatrix}\right]$, we have that $(1/n)H^*H$ is given by
\begin{scriptsize}
\[
\arraycolsep=1.0pt\def\arraystretch{1.0}
\left[
\begin{array}{cccccccccccccccc}
c&abi&-a^2&abi&-abj&-b^2k&a^2&-abi&-abj&-b^2k&-abj&-b^2k&a^2&-abi&-abj&-b^2k\\
-abi&d&-abi&-b^2&a^2k&-abj&abi&b^2&a^2k&-abj&a^2k&-abj&abi&b^2&a^2k&-abj\\
-a^2&abi&c&abi&-abj&-b^2k&-a^2&abi&abj&b^2k&-abj&-b^2k&-a^2&abi&abj&b^2k\\
-abi&-b^2&-abi&d&a^2k&-abj&-abi&-b^2&-a^2k&abj&a^2k&-abj&-abi&-b^2&-a^2k&abj\\
abj&-a^2k&abj&-a^2k&c&abi&b^2&abi&b^2&abi&a^2&-abi&-b^2&-abi&b^2&abi\\
b^2k&abj&b^2k&abj&-abi&d&-abi&a^2&-abi&a^2&abi&b^2&abi&-a^2&-abi&a^2\\
a^2&-abi&-a^2&abi&b^2&abi&c&abi&b^2&abi&b^2&abi&a^2&-abi&-b^2&-abi\\
abi&b^2&-abi&-b^2&-abi&a^2&-abi&d&-abi&a^2&-abi&a^2&abi&b^2&abi&-a^2\\
abj&-a^2k&-abj&a^2k&b^2&abi&b^2&abi&c&abi&-b^2&-abi&b^2&abi&a^2&-abi\\
b^2k&abj&-b^2k&-abj&-abi&a^2&-abi&a^2&-abi&d&abi&-a^2&-abi&a^2&abi&b^2\\
abj&-a^2k&abj&-a^2k&a^2&-abi&b^2&abi&-b^2&-abi&c&abi&-b^2&-abi&-b^2&-abi\\
b^2k&abj&b^2k&abj&abi&b^2&-abi&a^2&abi&-a^2&-abi&d& abi&-a^2&abi&-a^2\\
a^2&-abi&-a^2&abi&-b^2&-abi&a^2&-abi&b^2&abi&-b^2&-abi&c&abi&-b^2&-abi\\
abi&b^2&-abi&-b^2&abi&-a^2&abi&b^2&-abi&a^2&abi&-a^2&-abi&d&abi&-a^2\\
abj&-a^2k&-abj&a^2k&b^2&abi&-b^2&-abi&a^2&-abi&-b^2&-abi&-b^2&-abi&c&abi\\
b^2k&abj&-b^2k&-abj&-abi&a^2&abi&-a^2&abi&b^2&abi&-a^2&abi&-a^2&-abi&d
\end{array}
\right],
\]
\end{scriptsize}
where $c=a^2+2b^2$ and $d=2a^2+b^2$.
\end{example}

\begin{remark} The cases where no variable is shown indicate an orthogonal design in one variable. Any kind of Hadamard matrix can be thought of as an orthogonal design in one variable.  
\end{remark}

\begin{example}
Let $\zeta = e^{\frac{2\pi i}{3}}$, and consider the Butson Hadamard matrix B$(3,3)$ 

$$
H=
\begin{bmatrix}
1 & 1 & 1 \\
1 & \zeta & \zeta^2 \\
1 & \zeta^2 & \zeta
\end{bmatrix}.
$$
Using the logarithm matrix of H

$$
\begin{bmatrix}
0 & 0 & 0 \\
0 & 1 & 2 \\
0 & 2 & 1
\end{bmatrix}
$$
in the construction and obtain the logarithm matrix of a BG$(36,6)$ given by
\[
\arraycolsep=1.0pt\def\arraystretch{1.0}
\left[
\begin{array}{cccccc|ccccccccccccccc|ccccccccccccccc}
0&0&0&0&0&0&0&0&0&0&0&0&0&0&0&0&0&0&0&0&0&\bar{0}&\bar{0}&\bar{0}&\bar{0}&\bar{0}&\bar{0}&\bar{0}&\bar{0}&\bar{0}&\bar{0}&\bar{0}&\bar{0}&\bar{0}&\bar{0}&\bar{0}\\
0&0&0&0&0&0&2&2&2&1&1&1&0&0&0&2&2&2&1&1&1&\bar{2}&\bar{2}&\bar{2}&\bar{1}&\bar{1}&\bar{1}&\bar{0}&\bar{0}&\bar{0}&\bar{2}&\bar{2}&\bar{2}&\bar{1}&\bar{1}&\bar{1}\\
0&0&0&0&0&0&1&1&1&2&2&2&0&0&0&1&1&1&2&2&2&\bar{1}&\bar{1}&\bar{1}&\bar{2}&\bar{2}&\bar{2}&\bar{0}&\bar{0}&\bar{0}&\bar{1}&\bar{1}&\bar{1}&\bar{2}&\bar{2}&\bar{2}\\
0&0&0&0&0&0&0&0&0&0&0&0&\bar{0}&\bar{0}&\bar{0}&\bar{0}&\bar{0}&\bar{0}&\bar{0}&\bar{0}&\bar{0}&\bar{0}&\bar{0}&\bar{0}&\bar{0}&\bar{0}&\bar{0}&0&0&0&0&0&0&0&0&0\\
0&0&0&0&0&0&2&2&2&1&1&1&\bar{0}&\bar{0}&\bar{0}&\bar{2}&\bar{2}&\bar{2}&\bar{1}&\bar{1}&\bar{1}&\bar{2}&\bar{2}&\bar{2}&\bar{1}&\bar{1}&\bar{1}&0&0&0&2&2&2&1&1&1\\
0&0&0&0&0&0&1&1&1&2&2&2&\bar{0}&\bar{0}&\bar{0}&\bar{1}&\bar{1}&\bar{1}&\bar{2}&\bar{2}&\bar{2}&\bar{1}&\bar{1}&\bar{1}&\bar{2}&\bar{2}&\bar{2}&0&0&0&1&1&1&2&2&2\\ \hline
0&1&2&0&1&2&0&0&0&0&1&2&0&2&1&0&2&1&0&1&2&0&0&0&0&1&2&0&2&1&\bar{0}&\bar{2}&\bar{1}&\bar{0}&\bar{1}&\bar{2}\\
0&1&2&0&1&2&0&0&0&2&0&1&1&0&2&1&0&2&2&0&1&0&0&0&2&0&1&1&0&2&\bar{1}&\bar{0}&\bar{2}&\bar{2}&\bar{0}&\bar{1}\\
0&1&2&0&1&2&0&0&0&1&2&0&2&1&0&2&1&0&1&2&0&0&0&0&1&2&0&2&1&0&\bar{2}&\bar{1}&\bar{0}&\bar{1}&\bar{2}&\bar{0}\\
0&2&1&0&2&1&0&1&2&0&0&0&0&1&2&0&2&1&0&2&1&\bar{0}&\bar{1}&\bar{2}&0&0&0&0&1&2&0&2&1&\bar{0}&\bar{2}&\bar{1}\\
0&2&1&0&2&1&2&0&1&0&0&0&2&0&1&1&0&2&1&0&2&\bar{2}&\bar{0}&\bar{1}&0&0&0&2&0&1&1&0&2&\bar{1}&\bar{0}&\bar{2}\\
0&2&1&0&2&1&1&2&0&0&0&0&1&2&0&2&1&0&2&1&0&\bar{1}&\bar{2}&\bar{0}&0&0&0&1&2&0&2&1&0&\bar{2}&\bar{1}&\bar{0}\\
0&0&0&\bar{0}&\bar{0}&\bar{0}&0&2&1&0&1&2&0&0&0&0&1&2&0&2&1&\bar{0}&\bar{2}&\bar{1}&\bar{0}&\bar{1}&\bar{2}&0&0&0&0&1&2&0&2&1\\
0&0&0&\bar{0}&\bar{0}&\bar{0}&1&0&2&2&0&1&0&0&0&2&0&1&1&0&2&\bar{1}&\bar{0}&\bar{2}&\bar{2}&\bar{0}&\bar{1}&0&0&0&2&0&1&1&0&2\\
0&0&0&\bar{0}&\bar{0}&\bar{0}&2&1&0&1&2&0&0&0&0&1&2&0&2&1&0&\bar{2}&\bar{1}&\bar{0}&\bar{1}&\bar{2}&\bar{0}&0&0&0&1&2&0&2&1&0\\
0&1&2&\bar{0}&\bar{1}&\bar{2}&0&2&1&0&2&1&0&1&2&0&0&0&0&1&2&0&2&1&\bar{0}&\bar{2}&\bar{1}&\bar{0}&\bar{1}&\bar{2}&0&0&0&0&1&2\\
0&1&2&\bar{0}&\bar{1}&\bar{2}&1&0&2&1&0&2&2&0&1&0&0&0&2&0&1&1&0&2&\bar{1}&\bar{0}&\bar{2}&\bar{2}&\bar{0}&\bar{1}&0&0&0&2&0&1\\
0&1&2&\bar{0}&\bar{1}&\bar{2}&2&1&0&2&1&0&1&2&0&0&0&0&1&2&0&2&1&0&\bar{2}&\bar{1}&\bar{0}&\bar{1}&\bar{2}&\bar{0}&0&0&0&1&2&0\\
0&2&1&\bar{0}&\bar{2}&\bar{1}&0&1&2&0&2&1&0&2&1&0&1&2&0&0&0&0&1&2&0&2&1&\bar{0}&\bar{2}&\bar{1}&\bar{0}&\bar{1}&\bar{2}&0&0&0\\
0&2&1&\bar{0}&\bar{2}&\bar{1}&2&0&1&1&0&2&1&0&2&2&0&1&0&0&0&2&0&1&1&0&2&\bar{1}&\bar{0}&\bar{2}&\bar{2}&\bar{0}&\bar{1}&0&0&0\\
0&2&1&\bar{0}&\bar{2}&\bar{1}&1&2&0&2&1&0&2&1&0&1&2&0&0&0&0&1&2&0&2&1&0&\bar{2}&\bar{1}&\bar{0}&\bar{1}&\bar{2}&\bar{0}&0&0&0\\ \hline
\bar{0}&\bar{1}&\bar{2}&\bar{0}&\bar{1}&\bar{2}&0&0&0&0&1&2&0&2&1&\bar{0}&\bar{2}&\bar{1}&\bar{0}&\bar{1}&\bar{2}&0&0&0&0&1&2&0&2&1&0&2&1&0&1&2\\
\bar{0}&\bar{1}&\bar{2}&\bar{0}&\bar{1}&\bar{2}&0&0&0&2&0&1&1&0&2&\bar{1}&\bar{0}&\bar{2}&\bar{2}&\bar{0}&\bar{1}&0&0&0&2&0&1&1&0&2&1&0&2&2&0&1\\
\bar{0}&\bar{1}&\bar{2}&\bar{0}&\bar{1}&\bar{2}&0&0&0&1&2&0&2&1&0&\bar{2}&\bar{1}&\bar{0}&\bar{1}&\bar{2}&\bar{0}&0&0&0&1&2&0&2&1&0&2&1&0&1&2&0\\
\bar{0}&\bar{2}&\bar{1}&\bar{0}&\bar{2}&\bar{1}&\bar{0}&\bar{1}&\bar{2}&0&0&0&0&1&2&0&2&1&\bar{0}&\bar{2}&\bar{1}&0&1&2&0&0&0&0&1&2&0&2&1&0&2&1\\
\bar{0}&\bar{2}&\bar{1}&\bar{0}&\bar{2}&\bar{1}&\bar{2}&\bar{0}&\bar{1}&0&0&0&2&0&1&1&0&2&\bar{1}&\bar{0}&\bar{2}&2&0&1&0&0&0&2&0&1&1&0&2&1&0&2\\
\bar{0}&\bar{2}&\bar{1}&\bar{0}&\bar{2}&\bar{1}&\bar{1}&\bar{2}&\bar{0}&0&0&0&1&2&0&2&1&0&\bar{2}&\bar{1}&\bar{0}&1&2&0&0&0&0&1&2&0&2&1&0&2&1&0\\
\bar{0}&\bar{0}&\bar{0}&0&0&0&\bar{0}&\bar{2}&\bar{1}&\bar{0}&\bar{1}&\bar{2}&0&0&0&0&1&2&0&2&1&0&2&1&0&1&2&0&0&0&0&1&2&0&2&1\\
\bar{0}&\bar{0}&\bar{0}&0&0&0&\bar{1}&\bar{0}&\bar{2}&\bar{2}&\bar{0}&\bar{1}&0&0&0&2&0&1&1&0&2&1&0&2&2&0&1&0&0&0&2&0&1&1&0&2\\
\bar{0}&\bar{0}&\bar{0}&0&0&0&\bar{2}&\bar{1}&\bar{0}&\bar{1}&\bar{2}&\bar{0}&0&0&0&1&2&0&2&1&0&2&1&0&1&2&0&0&0&0&1&2&0&2&1&0\\
\bar{0}&\bar{1}&\bar{2}&0&1&2&0&2&1&\bar{0}&\bar{2}&\bar{1}&\bar{0}&\bar{1}&\bar{2}&0&0&0&0&1&2&0&2&1&0&2&1&0&1&2&0&0&0&0&1&2\\
\bar{0}&\bar{1}&\bar{2}&0&1&2&1&0&2&\bar{1}&\bar{0}&\bar{2}&\bar{2}&\bar{0}&\bar{1}&0&0&0&2&0&1&1&0&2&1&0&2&2&0&1&0&0&0&2&0&1\\
\bar{0}&\bar{1}&\bar{2}&0&1&2&2&1&0&\bar{2}&\bar{1}&\bar{0}&\bar{1}&\bar{2}&\bar{0}&0&0&0&1&2&0&2&1&0&2&1&0&1&2&0&0&0&0&1&2&0\\
\bar{0}&\bar{2}&\bar{1}&0&2&1&0&1&2&0&2&1&\bar{0}&\bar{2}&\bar{1}&\bar{0}&\bar{1}&\bar{2}&0&0&0&0&1&2&0&2&1&0&2&1&0&1&2&0&0&0\\
\bar{0}&\bar{2}&\bar{1}&0&2&1&2&0&1&1&0&2&\bar{1}&\bar{0}&\bar{2}&\bar{2}&\bar{0}&\bar{1}&0&0&0&2&0&1&1&0&2&1&0&2&2&0&1&0&0&0\\
\bar{0}&\bar{2}&\bar{1}&0&2&1&1&2&0&2&1&0&\bar{2}&\bar{1}&\bar{0}&\bar{1}&\bar{2}&\bar{0}&0&0&0&1&2&0&2&1&0&2&1&0&1&2&0&0&0&0\\
\end{array}
\right],
\]

where $\bar{a}$ corresponds to $-\zeta^a$, $a\in\{0,1,2\}$. 
\end{example}

\subsection{A relation between quaternion ETFs and complex ETFs}

For $q=z+wj\in\mathbb{H}$ where $z,w\in \mathbb{C}$, define
\begin{align*}
\text{Co}_1(z+wj)=z,\quad \text{Co}_2(z+wj)=\overline{w}.
\end{align*}
Define $[\cdot]_{\mathbb{C}}:\mathbb{H}^d\rightarrow \mathbb{C}^{2d}$ by
$$
[z+wj]_{\mathbb{C}}=\begin{pmatrix} z \\ \overline{w} \end{pmatrix}
$$
where $z,w\in \mathbb{C}^d$.

In \cite[Theorem~3.2]{waldron2020tight} it is shown that
a tight frame $V=(v_1,\ldots,v_n)$ for $\mathbb{H}^d$ corresponds to a tight frame for $\mathbb{C}^{2d}$ if and only if it satisfies
\begin{align}\label{eq:QETF}
\sum_{j,k=1}^n|\text{Co}_1(\langle v_j,v_k\rangle)|^2=\sum_{j,k=1}^n|\text{Co}_2(\langle v_j,v_k\rangle)|^2.
\end{align}

\begin{example}\label{ex:1}
Consider the first horizontal frame $H$ of the splittable QOD(16;8;8) constructed above. Upon setting $a=b=1$, we find that $(1/2)H^*H$ becomes
\[
\arraycolsep=1.0pt\def\arraystretch{1.0}
\left[
\begin{array}{cccccccccccccccc}
3&i&-&i&\bar{j}&\bar{k}&1&\bar{i}&\bar{j}&\bar{k}&\bar{j}&\bar{k}&1&\bar{i}&\bar{j}&\bar{k}\\
\bar{i}&3&\bar{i}&-&k&\bar{j}&i&1&k&\bar{j}&k&\bar{j}&i&1&k&\bar{j}\\
-&i&3&i&\bar{j}&\bar{k}&-&i&j&k&\bar{j}&\bar{k}&-&i&j&k\\
\bar{i}&-&\bar{i}&3&k&\bar{j}&\bar{i}&-&\bar{k}&j&k&\bar{j}&\bar{i}&-&\bar{k}&j\\
j&\bar{k}&j&\bar{k}&3&i&1&i&1&i&1&\bar{i}&-&\bar{i}&1&i\\
k&j&k&j&\bar{i}&3&\bar{i}&1&\bar{i}&1&i&1&i&-&\bar{i}&1\\
1&\bar{i}&-&i&1&i&3&i&1&i&1&i&1&\bar{i}&-&\bar{i}\\
i&1&\bar{i}&-&\bar{i}&1&\bar{i}&3&\bar{i}&1&\bar{i}&1&i&1&i&-\\
j&\bar{k}&\bar{j}&k&1&i&1&i&3&i&-&\bar{i}&1&i&1&\bar{i}\\
k&j&\bar{k}&\bar{j}&\bar{i}&1&\bar{i}&1&\bar{i}&3&i&-&\bar{i}&1&i&1\\
j&\bar{k}&j&\bar{k}&1&\bar{i}&1&i&-&\bar{i}&3&i&-&\bar{i}&-&\bar{i}\\
k&j&k&j&i&1&\bar{i}&1&i&-&\bar{i}&3&i&-&i&-\\
1&\bar{i}&-&i&-&\bar{i}&1&\bar{i}&1&i&-&\bar{i}&3&i&-&\bar{i}\\
i&1&\bar{i}&-&i&-&i&1&\bar{i}&1&i&-&\bar{i}&3&i&-\\
j&\bar{k}&\bar{j}&k&1&i&-&\bar{i}&1&\bar{i}&-&\bar{i}&-&\bar{i}&3&i\\
k&j&\bar{k}&\bar{j}&\bar{i}&1&i&-&i&1&i&-&i&-&\bar{i}&3\\
\end{array}
\right].
\]
It follows that $320 = \sum_{j,k=1}^{16}|\text{Co}_1(\langle v_j,v_k \rangle)|^2 \neq \sum_{j,k=1}^{16}|\text{Co}_2(\langle v_j,v_k \rangle)|^2 = 64$; whence, the quaternion frame associated with the horizontal frame $H$ is not given by a complex frame.
\end{example}

\begin{example}\label{ex:2}
Beginning with the restricted QOD(6;1;5) given by
\[
\arraycolsep=1.0pt\def\arraystretch{1.0}
\left[
\begin{array}{cccccc}
a&b&b&b&b&b\\
b&\bar{a}&bk&\bar{bk}&\bar{bk}&bk\\
b&bk&\bar{a}&bk&\bar{bk}&\bar{bk}\\
b&\bar{bk}&bk&\bar{a}&bk&\bar{bk}\\
b&\bar{bk}&\bar{bk}&bk&\bar{a}&bk\\
b&bk&\bar{bk}&\bar{bk}&bk&\bar{a}\\
\end{array}
\right],
\]
we use a complex Hadamard matrix of order six in order to construct a restricted splittable QOD(144;24,120). Again, taking $H$ to be the first horizontal frame, we find that for $(1/6)H^*H$ it follows that $29056 = \sum_{j,k=1}^{144}|\text{Co}_1(\langle v_j,v_k \rangle)|^2 \neq \sum_{j,k=1}^{144}|\text{Co}_2(\langle v_j,v_k \rangle)|^2 = 8960$. We then have a further quaternion frame not obtained from a complex frame.
\end{example}

\begin{problem} Is it possible to construct an infinite class of QOD $X_{16n^2}$ such that the equation~\eqref{eq:QETF} does not hold?
\end{problem}

\section{Unbiased orthogonal designs}

\indent Two Hadamard matrices $H_n$ and $K_n$ of order $n$ are said to be unbiased if $(1/\sqrt{n})H_nK_n^t$ is an Hadamard matrix. In this case, we can see that $n$ must be a square. The following was shown in \cite{KharSuda2}.
\begin{theorem}
Let $H = \left[\begin{smallmatrix} H_1^t & H_2^t \end{smallmatrix}\right]^t$ be a Hadamard matrix such that $H_1^tH_1 = \ell I_n + aA -a\bar{A}$. Then the following are equivalent:
\begin{enumerate}
\item $K := (1/2a)(H_1^tH_1 - H_2^tH_2)$ is an Hadamard matrix, and

\item $\ell = (n \pm \sqrt{n})/2$ and $a = \sqrt{n}/2$.
\end{enumerate}
\end{theorem}

Let $X_{4n^2}$ be as in \eqref{eq:X}. Then $X_{4n^2}$ becomes an Hadamard matrix upon replacing each indeterminate with unity. Moreover, the matrix is balancedly splittable with $b = -a = n$ and $\ell = 2n^2-n$. We have, therefore, that condition (ii) is satisfied, and we can construct an unbiased pair of Hadamard matrices. Since the vertical frames of $X_{4n^2}$ induce a stable split, however, we can in fact obtain more general result than this as follows.

In \cite{KharSuda1} unbiased orthogonal designs were introduced and are defined thus.

\begin{definition}
Let $X_n$ and $Y_n$ each be an OD$(n;s_1,s_2,\dots,s_u)$. If $X_nY_n^t = (\sigma/\sqrt{\alpha})W$, for some weighing matrix $W$ and some real number $\alpha$, then $X_n$ and $Y_n$ are said to be unbiased.
\end{definition}
Extending this, unbiasedness is defined for quaternionic orthogonal designs.
\begin{definition}
Let $X_n$ and $Y_n$ each be a QOD$(n;s_1,s_2,\dots,s_u)$. If $X_nY_n^* = (\sigma/\sqrt{\alpha})W$, for some weighing matrix $W$ and some real number $\alpha$, then $X_n$ and $Y_n$ are said to be unbiased.
\end{definition}

It is at this point that the stability of the vertical frames of the design becomes important, for the next result is a consequence of this property.

\begin{theorem}
Let $X_{4n^2}$ be as in \eqref{eq:X}, and let
\begin{align*}
V =
\begin{bmatrix}
-F \\ B \\ A
\end{bmatrix}
,\quad
U =
\begin{bmatrix}
G & F \\ E & A \\ -E & B
\end{bmatrix}
\end{align*}
so that $X_{4n^2} = [U \text{ } V]$. Take $Y_{4n^2} = [U \text{ } -V]$. Then $Y_{4n^2}$ is a QOD$(4n^2;4ns_1,4ns_2,\dots,4ns_u)$ and is unbiased to $X_{4n^2}$. In particular, $X_{4n^2}Y_{4n^2}^* = 2\sigma K_{4n^2}$, where $K_{4n^2}$ is an Hadamard matrix.
\end{theorem}
\begin{proof}
Clearly, $Y_{4n^2}$ is a QOD$(4n^2;4ns_1,4ns_2,\dots,4ns_u)$. Then $X_{4n^2}Y_{4n^2}^t = (UU^t - VV^t)$. Take $K_{4n^2} = (1/2\sigma)(UU^t - VV^t)$. We claim that $K_{4n^2}$ is Hadamard. Indeed,
$K_{4n^2}$ is a $(1,-1)$-matrix, and
\begin{align*}
K_{4n^2}K_{4n^2}^t &= K_{4n^2}^2 = \frac{1}{4\sigma^2}(UU^tUU^t + VV^tVV^t) \\
&= \frac{n}{\sigma}(UU^t + VV^t) = 4n^2I_{4n^2},
\end{align*}
where we have used the fact that since the off-diagonal entries of $VV^t$ are $\pm\sigma$, the off-diagonal entries of $UU^t$ must correspondingly be $\mp\sigma$.
\end{proof}
\begin{remark}
The resulting real Hadamard martrix $K_{4n^2}$ is not the Hadamard matrix obtained from $X_{4n^2}$ by replacing each indetermiante with unity, because $X_{4n^2}$ is not necessarily a real orthogonal design. 
\end{remark}
\begin{example}
Using the balancedly splittable OD$(16;8,8)$, we obtain the Hadamard matrix \\
\[\arraycolsep=1.0pt\def\arraystretch{1.0}
\left[
\begin{array}{cccccccccccccccc}
 1 & 1 & 1 & 1 & 1 & 1 & 1 & 1 & 1 & 1 & 1 & 1 & 1 & 1 & 1 & 1 \\
 1 & 1 & 1 & 1 & - & - & 1 & 1 & - & - & - & - & 1 & 1 & - & - \\
 1 & 1 & 1 & 1 & 1 & 1 & - & - & - & - & 1 & 1 & - & - & - & - \\
 1 & 1 & 1 & 1 & - & - & - & - & 1 & 1 & - & - & - & - & 1 & 1 \\
 1 & - & 1 & - & 1 & 1 & 1 & - & 1 & - & 1 & 1 & 1 & - & - & 1 \\
 1 & - & 1 & - & 1 & 1 & - & 1 & - & 1 & 1 & 1 & - & 1 & 1 & - \\
 1 & 1 & - & - & 1 & - & 1 & 1 & 1 & - & - & 1 & 1 & 1 & 1 & - \\
 1 & 1 & - & - & - & 1 & 1 & 1 & - & 1 & 1 & - & 1 & 1 & - & 1 \\
 1 & - & - & 1 & 1 & - & 1 & - & 1 & 1 & 1 & - & - & 1 & 1 & 1 \\
 1 & - & - & 1 & - & 1 & - & 1 & 1 & 1 & - & 1 & 1 & - & 1 & 1 \\
 1 & - & 1 & - & 1 & 1 & - & 1 & 1 & - & 1 & 1 & - & 1 & - & 1 \\
 1 & - & 1 & - & 1 & 1 & 1 & - & - & 1 & 1 & 1 & 1 & - & 1 & - \\
 1 & 1 & - & - & 1 & - & 1 & 1 & - & 1 & - & 1 & 1 & 1 & - & 1 \\
 1 & 1 & - & - & - & 1 & 1 & 1 & 1 & - & 1 & - & 1 & 1 & 1 & - \\
 1 & - & - & 1 & - & 1 & 1 & - & 1 & 1 & - & 1 & - & 1 & 1 & 1 \\
 1 & - & - & 1 & 1 & - & - & 1 & 1 & 1 & 1 & - & 1 & - & 1 & 1 \\
\end{array}
\right].
\] \\
\end{example}

\begin{example}
Take the BH$(36,6)$ in Example 4, whose form is $X = [U^t \text{ } V^t]^t$ by taking the horizontal frames, and let $Y = [U^t \text{ } -V^t]^t$. Then $(1/n)XY^*$ is a Butson B$(36,6)$ Hadamard matrix $K = (U^tU-V^t V)$ given by
\[
\arraycolsep=1.0pt\def\arraystretch{1.0}
\left[
\begin{array}{cccccccccccccccccccccccccccccccccccc}
0&0&0&0&0&0&0&0&0&0&0&0&0&0&0&0&0&0&0&0&0&0&0&0&0&0&0&0&0&0&0&0&0&0&0&0\\
0&0&0&0&0&0&2&2&2&1&1&1&0&0&0&2&2&2&1&1&1&2&2&2&1&1&1&0&0&0&2&2&2&1&1&1\\
0&0&0&0&0&0&1&1&1&2&2&2&0&0&0&1&1&1&2&2&2&1&1&1&2&2&2&0&0&0&1&1&1&2&2&2\\
0&0&0&0&0&0&0&0&0&0&0&0&\bar{0}&\bar{0}&\bar{0}&\bar{0}&\bar{0}&\bar{0}&\bar{0}&\bar{0}&\bar{0}&0&0&0&0&0&0&\bar{0}&\bar{0}&\bar{0}&\bar{0}&\bar{0}&\bar{0}&\bar{0}&\bar{0}&\bar{0}\\
0&0&0&0&0&0&2&2&2&1&1&1&\bar{0}&\bar{0}&\bar{0}&\bar{2}&\bar{2}&\bar{2}&\bar{1}&\bar{1}&\bar{1}&2&2&2&1&1&1&\bar{0}&\bar{0}&\bar{0}&\bar{2}&\bar{2}&\bar{2}&\bar{1}&\bar{1}&\bar{1}\\
0&0&0&0&0&0&1&1&1&2&2&2&\bar{0}&\bar{0}&\bar{0}&\bar{1}&\bar{1}&\bar{1}&\bar{2}&\bar{2}&\bar{2}&1&1&1&2&2&2&\bar{0}&\bar{0}&\bar{0}&\bar{1}&\bar{1}&\bar{1}&\bar{2}&\bar{2}&\bar{2}\\
0&1&2&0&1&2&0&0&0&0&2&1&0&1&2&0&1&2&0&2&1&\bar{0}&\bar{0}&\bar{0}&0&2&1&\bar{0}&\bar{1}&\bar{2}&0&1&2&\bar{0}&\bar{2}&\bar{1}\\
0&1&2&0&1&2&0&0&0&1&0&2&2&0&1&2&0&1&1&0&2&\bar{0}&\bar{0}&\bar{0}&1&0&2&\bar{2}&\bar{0}&\bar{1}&2&0&1&\bar{1}&\bar{0}&\bar{2}\\
0&1&2&0&1&2&0&0&0&2&1&0&1&2&0&1&2&0&2&1&0&\bar{0}&\bar{0}&\bar{0}&2&1&0&\bar{1}&\bar{2}&\bar{0}&1&2&0&\bar{2}&\bar{1}&\bar{0}\\
0&2&1&0&2&1&0&2&1&0&0&0&0&2&1&0&1&2&0&1&2&\bar{0}&\bar{2}&\bar{1}&\bar{0}&\bar{0}&\bar{0}&0&2&1&\bar{0}&\bar{1}&\bar{2}&0&1&2\\
0&2&1&0&2&1&1&0&2&0&0&0&1&0&2&2&0&1&2&0&1&\bar{1}&\bar{0}&\bar{2}&\bar{0}&\bar{0}&\bar{0}&1&0&2&\bar{2}&\bar{0}&\bar{1}&2&0&1\\
0&2&1&0&2&1&2&1&0&0&0&0&2&1&0&1&2&0&1&2&0&\bar{2}&\bar{1}&\bar{0}&\bar{0}&\bar{0}&\bar{0}&2&1&0&\bar{1}&\bar{2}&\bar{0}&1&2&0\\
0&0&0&\bar{0}&\bar{0}&\bar{0}&0&1&2&0&2&1&0&0&0&0&2&1&0&1&2&0&1&2&\bar{0}&\bar{2}&\bar{1}&\bar{0}&\bar{0}&\bar{0}&0&2&1&\bar{0}&\bar{1}&\bar{2}\\
0&0&0&\bar{0}&\bar{0}&\bar{0}&2&0&1&1&0&2&0&0&0&1&0&2&2&0&1&2&0&1&\bar{1}&\bar{0}&\bar{2}&\bar{0}&\bar{0}&\bar{0}&1&0&2&\bar{2}&\bar{0}&\bar{1}\\
0&0&0&\bar{0}&\bar{0}&\bar{0}&1&2&0&2&1&0&0&0&0&2&1&0&1&2&0&1&2&0&\bar{2}&\bar{1}&\bar{0}&\bar{0}&\bar{0}&\bar{0}&2&1&0&\bar{1}&\bar{2}&\bar{0}\\
0&1&2&\bar{0}&\bar{1}&\bar{2}&0&1&2&0&1&2&0&2&1&0&0&0&0&2&1&\bar{0}&\bar{1}&\bar{2}&0&1&2&\bar{0}&\bar{2}&\bar{1}&\bar{0}&\bar{0}&\bar{0}&0&2&1\\
0&1&2&\bar{0}&\bar{1}&\bar{2}&2&0&1&2&0&1&1&0&2&0&0&0&1&0&2&\bar{2}&\bar{0}&\bar{1}&2&0&1&\bar{1}&\bar{0}&\bar{2}&\bar{0}&\bar{0}&\bar{0}&1&0&2\\
0&1&2&\bar{0}&\bar{1}&\bar{2}&1&2&0&1&2&0&2&1&0&0&0&0&2&1&0&\bar{1}&\bar{2}&\bar{0}&1&2&0&\bar{2}&\bar{1}&\bar{0}&\bar{0}&\bar{0}&\bar{0}&2&1&0\\
0&2&1&\bar{0}&\bar{2}&\bar{1}&0&2&1&0&1&2&0&1&2&0&2&1&0&0&0&0&2&1&\bar{0}&\bar{1}&\bar{2}&0&1&2&\bar{0}&\bar{2}&\bar{1}&\bar{0}&\bar{0}&\bar{0}\\
0&2&1&\bar{0}&\bar{2}&\bar{1}&1&0&2&2&0&1&2&0&1&1&0&2&0&0&0&1&0&2&\bar{2}&\bar{0}&\bar{1}&2&0&1&\bar{1}&\bar{0}&\bar{2}&\bar{0}&\bar{0}&\bar{0}\\
0&2&1&\bar{0}&\bar{2}&\bar{1}&2&1&0&1&2&0&1&2&0&2&1&0&0&0&0&2&1&0&\bar{1}&\bar{2}&\bar{0}&1&2&0&\bar{2}&\bar{1}&\bar{0}&\bar{0}&\bar{0}&\bar{0}\\
0&1&2&0&1&2&\bar{0}&\bar{0}&\bar{0}&\bar{0}&\bar{2}&\bar{1}&0&1&2&\bar{0}&\bar{1}&\bar{2}&0&2&1&0&0&0&\bar{0}&\bar{2}&\bar{1}&\bar{0}&\bar{1}&\bar{2}&\bar{0}&\bar{1}&\bar{2}&\bar{0}&\bar{2}&\bar{1}\\
0&1&2&0&1&2&\bar{0}&\bar{0}&\bar{0}&\bar{1}&\bar{0}&\bar{2}&2&0&1&\bar{2}&\bar{0}&\bar{1}&1&0&2&0&0&0&\bar{1}&\bar{0}&\bar{2}&\bar{2}&\bar{0}&\bar{1}&\bar{2}&\bar{0}&\bar{1}&\bar{1}&\bar{0}&\bar{2}\\
0&1&2&0&1&2&\bar{0}&\bar{0}&\bar{0}&\bar{2}&\bar{1}&\bar{0}&1&2&0&\bar{1}&\bar{2}&\bar{0}&2&1&0&0&0&0&\bar{2}&\bar{1}&\bar{0}&\bar{1}&\bar{2}&\bar{0}&\bar{1}&\bar{2}&\bar{0}&\bar{2}&\bar{1}&\bar{0}\\
0&2&1&0&2&1&0&2&1&\bar{0}&\bar{0}&\bar{0}&\bar{0}&\bar{2}&\bar{1}&0&1&2&\bar{0}&\bar{1}&\bar{2}&\bar{0}&\bar{2}&\bar{1}&0&0&0&\bar{0}&\bar{2}&\bar{1}&\bar{0}&\bar{1}&\bar{2}&\bar{0}&\bar{1}&\bar{2}\\
0&2&1&0&2&1&1&0&2&\bar{0}&\bar{0}&\bar{0}&\bar{1}&\bar{0}&\bar{2}&2&0&1&\bar{2}&\bar{0}&\bar{1}&\bar{1}&\bar{0}&\bar{2}&0&0&0&\bar{1}&\bar{0}&\bar{2}&\bar{2}&\bar{0}&\bar{1}&\bar{2}&\bar{0}&\bar{1}\\
0&2&1&0&2&1&2&1&0&\bar{0}&\bar{0}&\bar{0}&\bar{2}&\bar{1}&\bar{0}&1&2&0&\bar{1}&\bar{2}&\bar{0}&\bar{2}&\bar{1}&\bar{0}&0&0&0&\bar{2}&\bar{1}&\bar{0}&\bar{1}&\bar{2}&\bar{0}&\bar{1}&\bar{2}&\bar{0}\\
0&0&0&\bar{0}&\bar{0}&\bar{0}&\bar{0}&\bar{1}&\bar{2}&0&2&1&\bar{0}&\bar{0}&\bar{0}&\bar{0}&\bar{2}&\bar{1}&0&1&2&\bar{0}&\bar{1}&\bar{2}&\bar{0}&\bar{2}&\bar{1}&0&0&0&\bar{0}&\bar{2}&\bar{1}&\bar{0}&\bar{1}&\bar{2}\\
0&0&0&\bar{0}&\bar{0}&\bar{0}&\bar{2}&\bar{0}&\bar{1}&1&0&2&\bar{0}&\bar{0}&\bar{0}&\bar{1}&\bar{0}&\bar{2}&2&0&1&\bar{2}&\bar{0}&\bar{1}&\bar{1}&\bar{0}&\bar{2}&0&0&0&\bar{1}&\bar{0}&\bar{2}&\bar{2}&\bar{0}&\bar{1}\\
0&0&0&\bar{0}&\bar{0}&\bar{0}&\bar{1}&\bar{2}&\bar{0}&2&1&0&\bar{0}&\bar{0}&\bar{0}&\bar{2}&\bar{1}&\bar{0}&1&2&0&\bar{1}&\bar{2}&\bar{0}&\bar{2}&\bar{1}&\bar{0}&0&0&0&\bar{2}&\bar{1}&\bar{0}&\bar{1}&\bar{2}&\bar{0}\\
0&1&2&\bar{0}&\bar{1}&\bar{2}&0&1&2&\bar{0}&\bar{1}&\bar{2}&0&2&1&\bar{0}&\bar{0}&\bar{0}&\bar{0}&\bar{2}&\bar{1}&\bar{0}&\bar{1}&\bar{2}&\bar{0}&\bar{1}&\bar{2}&\bar{0}&\bar{2}&\bar{1}&0&0&0&\bar{0}&\bar{2}&\bar{1}\\
0&1&2&\bar{0}&\bar{1}&\bar{2}&2&0&1&\bar{2}&\bar{0}&\bar{1}&1&0&2&\bar{0}&\bar{0}&\bar{0}&\bar{1}&\bar{0}&\bar{2}&\bar{2}&\bar{0}&\bar{1}&\bar{2}&\bar{0}&\bar{1}&\bar{1}&\bar{0}&\bar{2}&0&0&0&\bar{1}&\bar{0}&\bar{2}\\
0&1&2&\bar{0}&\bar{1}&\bar{2}&1&2&0&\bar{1}&\bar{2}&\bar{0}&2&1&0&\bar{0}&\bar{0}&\bar{0}&\bar{2}&\bar{1}&\bar{0}&\bar{1}&\bar{2}&\bar{0}&\bar{1}&\bar{2}&\bar{0}&\bar{2}&\bar{1}&\bar{0}&0&0&0&\bar{2}&\bar{1}&\bar{0}\\
0&2&1&\bar{0}&\bar{2}&\bar{1}&\bar{0}&\bar{2}&\bar{1}&0&1&2&\bar{0}&\bar{1}&\bar{2}&0&2&1&\bar{0}&\bar{0}&\bar{0}&\bar{0}&\bar{2}&\bar{1}&\bar{0}&\bar{1}&\bar{2}&\bar{0}&\bar{1}&\bar{2}&\bar{0}&\bar{2}&\bar{1}&0&0&0\\
0&2&1&\bar{0}&\bar{2}&\bar{1}&\bar{1}&\bar{0}&\bar{2}&2&0&1&\bar{2}&\bar{0}&\bar{1}&1&0&2&\bar{0}&\bar{0}&\bar{0}&\bar{1}&\bar{0}&\bar{2}&\bar{2}&\bar{0}&\bar{1}&\bar{2}&\bar{0}&\bar{1}&\bar{1}&\bar{0}&\bar{2}&0&0&0\\
0&2&1&\bar{0}&\bar{2}&\bar{1}&\bar{2}&\bar{1}&\bar{0}&1&2&0&\bar{1}&\bar{2}&\bar{0}&2&1&0&\bar{0}&\bar{0}&\bar{0}&\bar{2}&\bar{1}&\bar{0}&\bar{1}&\bar{2}&\bar{0}&\bar{1}&\bar{2}&\bar{0}&\bar{2}&\bar{1}&\bar{0}&0&0&0\\
\end{array}
\right].
\]
\end{example}

\section{Concluding remarks}
As a summary, we have obtained the following results in the paper.
\begin{enumerate}
\item A recursive method to obtain a quaternionic full orthogonal design of order $4n^2$ from a quaternion full orthogonal design of order $n$ (Theorem~\ref{thm:qod}).
\item The constructed orthogonal designs in the above are balancedly splittable (Theorem~\ref{thm:bqod}).
\item Balancedly splittable QODs yield quaternionic equiangular tight frames (Proposition~\ref{prop:ETF}) and the constructed QOD of small orders $16,144$ yields quaternionic ETFs not obtainable from complex ETFs (Examples~\ref{ex:1},\ref{ex:2})
\item It is shown in \cite{KharSuda2} that there is no balancedly splittable Hadamard matrix of order $4n^2$, $n$ odd. Therefore, there is no full, balancedly splittable orthogonal design of order $4n^2$, $n$ odd. The existence of balancedly splittable real orthogonal designs of order $4n^2$ is shown here for any $n$, which is the order of a full orthogonal design. The existence of balancedly splittable full orthogonal designs of order $16n^2$, $n$ odd remains open for the real case. The first open case in the real case is $n=144$.\end{enumerate}

\section*{Acknowledgments.}
The authors would like to thank the anonymous referees for valuable comments. 
Hadi Kharaghani is supported by the Natural Sciences and
Engineering  Research Council of Canada (NSERC).  Sho Suda is supported by JSPS KAKENHI Grant Number 18K03395.

\begin{rezabib}

\bib{CCHT2018}{article}{    
author = {Jameson Cahill}, 
author={Peter G. Casazza}, 
author={John I. Haas},
author={Janet Tremain},
 title = {Constructions of biangular tight frames and their
              relationships with equiangular tight frames},
 BOOKTITLE = {Frames and harmonic analysis, AMS Contemporary Math.},
    SERIES = {Contemp. Math.},
    VOLUME = {706},
     PAGES = {1--19},
 PUBLISHER = {Amer. Math. Soc., Providence, RI},
      date= {(2018)}
}
\bib{FKS2018}{article}{
 author ={Kai Fender},
 author={Kharaghani, Hadi},
 author={Suda, Sho},
title = {On a class of quaternary complex Hadamard matrices},
journal = {Discrete Mathematics},
volume = {341},
number = {2},
pages = {421 - 426},
year = {2018},
}

\bib{Khar}{article}{
   author={Hadi Kharaghani, },
   title={New class of weighing matrices},
   journal={Ars Combin.},
   volume={19},
   date={1985},
   pages={69--72},
   issn={0381-7032},
}

\bib{KharSuda1}{article}{
   author={Kharaghani, Hadi},
   author={Suda, Sho},
   title={Unbiased orthogonal designs},
   journal={Des. Codes Cryptogr.},
   volume={86},
   date={2018},
   number={7},
   pages={1573--1588},
   issn={0925-1022},
}

\bib{KharSuda2}{article}{
   author={Kharaghani, Hadi},
   author={Suda, Sho},
   title={Balancedly splittable Hadamard matrices},
   journal={Discrete Math.},
   volume={342},
   date={2019},
   number={2},
   pages={546--561},
   issn={0012-365X},
}

\bib{Seb}{book}{
   author={Seberry, Jennifer},
   title={Orthogonal designs},
   note={Hadamard matrices, quadratic forms and algebras;
   Revised and updated edition of the 1979 original [ MR0534614]},
   publisher={Springer, Cham},
   date={2017},
   pages={xxiii+453},
   isbn={978-3-319-59031-8},
   isbn={978-3-319-59032-5},
   doi={10.1007/978-3-319-59032-5},
}
\bib{waldron2020tight}{article}{
      title={Tight frames over the quaternions and equiangular lines},
      author={Shayne Waldron},
      journal = {arXiv:2006.06126},
      archivePrefix={arXiv},
      primaryClass={math.FA}
}
\end{rezabib}
\end{document}